\documentclass[preprint,12pt]{elsarticle}

 \usepackage{graphicx}

\usepackage{amssymb}
 \usepackage{amsthm}

\usepackage{amsmath}
\usepackage{bbm}

\newcommand{\rr}{\mathbb R}

\newcommand{\nn}{\mathbb N}

\newcommand{\EE}{\mathcal E}
\newcommand{\NN}{\mathcal N}
\newcommand{\FF}{\mathcal F}
\newcommand{\pp}{\mathbb P}
\newcommand{\ppx}{\mathbb P^x}
\newcommand{\ppz}{\mathbb P^z}
\newcommand{\ee}{\mathbb E}
\newcommand{\eex}{\mathbb E^x}
\newcommand{\eey}{\mathbb E^y}
\newcommand{\eez}{\mathbb E^z}

\newtheorem{assumption}{Assumption}
\newtheorem{prop}{Proposition}
\newtheorem{theorem}{Theorem}

\newtheorem{lemma}{Lemma}

\begin{document}

\begin{frontmatter}


 \author[bmw]{Brian M. Whitehead \corref{cor1}\fnref[ }
 \ead{whiteheadb@easternct.edu}

 \cortext[cor1]{Corresponding Author}
 \address[bmw]{Eastern Connecticut State University, 83 Windham St., Willimantic, CT 06226, U.S.A.}
 \fntext[]{Telephone: 860-465-4617}

\title{Path Results for Symmetric Jump Processes}


\author{}

\address{}

\begin{abstract}
We consider a class of jump processes in euclidean space which are associated to a certain non-local symmetric Dirichlet form.  We prove a lower bound on the occupation times of sets, and that a support theorem holds for these processes. 

\end{abstract}


\end{frontmatter}


\section{Introduction}\label{intro}

In this paper, we will consider a class of symmetric Markov processes of pure jump type in $\rr^d$ associated to the Dirichlet form $$\EE(f,f) = \int_{\rr^d} \int_{\rr^d}(f(y) - f(x))^2 J(x,y) dx \thinspace dy.$$  Here $J(x,y)$ is the jump kernel, and controls the intensity of the number of jumps from $x$ to $y$.  We will assume that these processes $X_t$ have no jumps of size larger than one, and that $J(x,y) \asymp |y-x|^{-d-\alpha}$ for some $\alpha \in (0,2)$ when $|y-x|<1$.  Here, we will prove a lower bound on the occupation times of sets and that a support theorem holds.  

These symmetric processes are of interest, since they arise when studying models of financial markets (see \cite{ss}).
 In \cite{bbck}, Barlow, Bass, Chen, and Kassmann gave some heat kernel estimates for such processes and showed a parabolic Harnack inequality, and in \cite{bkk}, Bass, Kassmann, and Kumagai gave further results regarding exit probabilities and the H\"{o}lder continuity of harmonic functions and of heat kernels.  We will use some estimates  from \cite{bbck} as part of our proofs.

Given the heat kernel estimates of \cite{bbck}, it is not difficult to show that $X$ will hit any set of positive Lebesgue measure with positive probability.  In Theorem \ref{t44}, we are able to extend this result by showing that $X$ can be expected to spend a positive amount of time in sets of positive Lebesgue measure.  In particular, we define the occupation time of a set $B$ to be 
\begin{equation}
\eex \int_0^\tau 1_B(X_s)ds,
\end{equation} 
where $\tau$ is the first time we leave some ball in $\rr^d$ containing the set $B$ and $|B|$ is the Lebesgue measure of $B$.  We show that there exists a nondecreasing function $\varphi : (0,1) \to (0,1)$ such that if $x \in Q(x_0, R/2)$ and $B \subseteq Q(x_0,R)$, then
$$ \eex \int_0^{\tau_{Q(x_0,R)}}1_B(X_s) ds \geq \varphi(|B|/\rr^d). $$ 
Here $Q(x_0,r)$ denotes the cube centered at $x_0$ with side length $r$.
Such a result was shown for a class of stable-like jump processes in \cite{bmw}.

 A support theorem is a result which states that there will be some positive probability that the processes we are considering will not stray too far from the image of any given continuous map $\varphi \colon [0,t_0] \to \rr^d$.  That is, if we fix $\varepsilon > 0$, and let $\varphi(0) = x_0$, then there exists $c_1 > 0$ depending on $\varphi$, $\varepsilon$, and $t_0$ such that
$$\pp^{x_0}\left(\displaystyle \sup_{s\leq t_0} |X_s - \varphi(s)| < \varepsilon \right) > c_1.$$
    Support theorems proven in other contexts have been useful tools in further proofs. In \cite{bc2} and \cite{bmw}, such  theorems have been shown for different classes of jump processes. 

Section 2 contains some preliminaries and states some results from \cite{bbck}, section 3 uses these results in order to obtain facts about the exit times of these processes, and in order to show that these processes will hit sets having positive Lebesgue measure with positive probability.  In section 4 we prove our main theorem concerning occupation times, and section 5 consists of the proof of the support theorem.

\section{Preliminaries}

We will consider the non-local symmetric Dirichlet form ($\EE$, $\FF$) given by 
\begin{align} \label{ee}
\EE(f,f) &= \int_{\rr^d} \int_{\rr^d}(f(y) - f(x))^2 J(x,y) dx \thinspace dy, \\
 \FF &= \overline{C_c^1(\rr^d)}^{\EE_1}, 
\end{align}
where $\EE_1(f,f) := \EE(f,f) + \|f\|_2^2$, $C_c^1(\rr^d)$ denotes the space of $C^1$ functions on $\rr^d$ with compact support, and $\FF$ is the closure of $C_c^1(\rr^d)$ with respect to the metric $\EE_1(f,f)^{1/2}$.  We make the following assumptions on the jump kernel $J(x,y)$.

\begin{assumption}\label{a21b}
a) $J(x,y) = J(y,x)$ for all $x$ and $y$. \newline
b) $J(x,y) = 0$ for $|x - y| \geq 1$.  \newline
c)  There exist $\alpha \in (0,2)$ and positive constants $\kappa_1$ and $\kappa_2$, such that if $|x-y| < 1$, 
$$\kappa_1|y-x|^{-d-\alpha} \leq J(x,y) \leq  \kappa_2|y-x|^{-d-\alpha}.$$
\end{assumption}
We observe that none of these conditions impose any sort of continuity on the jump kernel $J$, and that jump intensities can depend on both the position of the process and the direction of the jump. 

Theorems 1.1 and 1.2 of \cite{bbck} state that under the above assumption there is a conservative Hunt process $X$ associated with the Dirichlet form $(\EE, \FF)$, and that this process has a symmetric transition density function $p(t,x,y)$ with respect to Lebesgue measure on $\rr^d$. This process $X$ has state space $\rr^d \setminus \NN$, where $\NN$ is a subset of $\rr^d$ that has zero capacity with respect to $(\EE, \FF)$.  A precise definition of capacity can be found in \cite{fot}.  It follows that this properly exceptional set $\NN$ has zero Lebesgue measure.   

The transition density function $p(t,x,y)$ is often called the heat kernel corresponding to $(\EE, \FF)$.  Let $B(x,r)$ denote the open ball of radius $r$ centered at $x$ and define $p^D(t,x,y)$ to be the transition density for the subprocess $X^D$ killed upon exiting the ball $D$.
  Results shown in \cite{bbck} provide us with upper and lower bounds for the heat kernel, and will be very useful in this paper.  These results are summarized in the following theorem, which  provides us with on-diagonal and off-diagonal upper bounds and a lower bound.

\begin{theorem}\label{t22}\label{p23b}\label{t25}
Suppose Assumption \ref{a21b} is satisfied.

a) There exists a properly exceptional set $\NN$ of $X$ and positive constants $c_1$ and $c_2$ (depending on the constants in Assumption \ref{a21b}) such that 
\begin{equation}
p(t,x,y) \leq c_1t^{-d/\alpha}e^{c_2t}
\end{equation}
for every $t > 0$ and $x,y \in \rr^d \setminus \NN$.

b) There exists a properly exceptional set $\NN$ of $X$ and positive constants $c_3$ and $c_4$ (depending on the constants in Assumption \ref{a21b}) such that 
\begin{equation}
p(t,x,y) \leq c_3te^{-c_4|x-y|}
\end{equation} 
for $t \in (0,2]$, and $x,y \in \rr^d \setminus \NN$ such that $|x-y| \geq \frac{1}{16}$.

c)  Let $y_0 \in \rr^d$, $T > 1/2$, and $\delta \in (0, 1/2)$.  Let $R > 0$ and $B = B(y_0, R)$, and take $\epsilon \in (0,1)$.  There exists a properly exceptional set $\NN$ and a positive constant $C$ that depends on $R$, $T$, $\alpha$, $\kappa_1$, $\kappa_2$, $\delta$, and $\epsilon$, but not on $y_0$, such that for all $t \in [\delta, T]$,
\begin{equation}
p^B(t,x,y) \geq C
\end{equation}
for every $(x,y) \in (B(y_0, \epsilon R) \setminus \NN) \times (B(y_0, \epsilon R) \setminus \NN)$.  
\end{theorem}

Throughout this paper, we denote by $Q(x,r)$ the cube of side length $r$ centered at $x$.  $|A|$ will denote the Lebesgue measure of $A$.  We denote the hitting and exit times of set $A$ respectively, by $T_A = \inf\{t>0:X_t \in A \}$ and $\tau_A = \inf\{t>0:X_t \notin A \}$.
We write $X_{t-} = \lim_{s \uparrow t} X_s$, and $\Delta X_t = X_t - X_{t-}$.
The letter $c$ with subscripts will denote various positive constants with unimportant values, which will depend on the constants in Assumption \ref{a21b} along with other dependencies that will be explicitly mentioned.
Let $\NN$ be an exceptional set of $X$ such that all parts of Theorem \ref{t22}  hold with respect to this set $\NN$.  We assume that Assumption \ref{a21b} holds throughout.

\section{Exiting Times}

In this section, we will obtain some preliminary results concerning $X$ which will be useful in proving our main theorems.  We will first get an upper bound for the heat kernel which is an improvement over Theorem \ref{t22} a) when $t$ is large.

\begin{prop}\label{p31}
We have that 
\begin{equation}
p(t,x,y) \leq c_1(t^{-d/\alpha} \wedge 1)
\end{equation}
for every $t > 0$ and $x,y \in \rr^d \setminus \NN.$
\end{prop}

\begin{proof}

Theorem \ref{t22} a) implies that
\begin{equation}\label{tsmall}
p(t,x,y) \leq c_2t^{-d/\alpha}, \qquad t \leq 1.
\end{equation}

Now suppose that $1<t\leq 3/2$.  By \eqref{tsmall}, there exists a constant $c_3$, such that $p(t-1/2,x,y) \leq c_3$.  Theorem 3.1 of \cite{bbck} states that we can write
$$p(r+s,x,y) = \int p(r,x,z)p(s, z, y)dz,$$
for every $x,y \in \rr^d$ and $r,s > 0$.  Therefore, we have that 
\begin{equation*}
p(t,x,y) = \int p(1/2,x,z)p(t-1/2,z,y)dz 
\leq c_3 \int p(1/2,x,z)dz 
= c_3.
\end{equation*}

Similarly, if $3/2 \le t \le 2$, then $p(t-{1},x,y) \leq c_3$, and 
\begin{equation*}
p(t,x,y) = \int p(1,x,z)p(t-1,z,y)dz 
\leq c_3 \int p(1,x,z)dz 
= c_3.
\end{equation*}

By induction, it suffices to show that if $p(t,x,y) \leq c_3$ whenever $1 < t \leq k$, then $p(t,x,y) \leq c_3$ for every $t \in (k, k+1]$.  Suppose that $k < t \leq k+1$.  Then we have that
\begin{equation*}
p(t,x,y) = \int p(1,x,z)p(t-1,z,y)dz 
\leq c_3 \int p(1,x,z)dz 
= c_3
\end{equation*}
by our inductive assumption, so in fact 
$$p(t,x,y) \leq c_3,\qquad t>1.$$
This, along with \eqref{tsmall}, gives us our desired result. 
\end{proof}

We now use our known heat kernel estimates to show that these processes will not leave a given ball (or cube) too quickly or too slowly.

\begin{prop}\label{p32}
Let $\varepsilon < 3/4$ and $r > 0$.  There exists $c_1$ depending on $\varepsilon$ and $r$ such that if $x_0 \in \rr^d$ and $r > 0$, then  $$ \inf_{z \in (B(x_0,\varepsilon r) \setminus \NN)} \ee^z \tau_{B(x_0,r)} \geq c_1.$$
\end{prop}

\begin{proof}
We denote $B(x_0,r)$ by $B$, $B(x_0, \varepsilon r)$ by $B'$, and $\tau_{B}$ by $\tau$, and we observe that for any $t$,
\begin{equation}\label{tau}
\ee^{x}\tau \geq \eex [ \tau; \tau > t]  \geq t\ppx(\tau > t).
\end{equation}

Let $z \in B(x_0,\varepsilon r) \setminus \NN$.  Then

\begin{align*}
 \pp^{z}(\tau >t) &= \pp^{z}(\sup_{s \leq t} |X_s - x_0| < r) \\
&= \int_{B} p^B(t,z,y)\thinspace dy \\ 
&\geq \int_{B'\setminus\NN} p^{B}(t,z,y) \thinspace dy. \\
\end{align*}

We now apply Theorem \ref{t25} c), and obtain that for all $t \in [1/4,1]$ and every $x,y \in B(x_0, \varepsilon r) \setminus \NN$, 
$$ p^{B}(t,x,y) \geq c_2.$$
Therefore, we have that 
 $$\pp^{z}(\tau > t) \geq c_2 |B'| = c_3.$$

Hence, by \eqref{tau} in the case where $t = 1$, we get that
$$ \eez \tau \geq c_3.$$

We further note that this constant $c_3$ does not depend on the location of $z$ inside $B(x_0, \varepsilon r) \setminus \NN$, so this statement holds for the infimum of such $z$.
\end{proof}
\begin{prop}\label{p33b}
There exists $c_1$ such that $$\sup_{z \in (B(x_0,r) \setminus \NN)} \ee^z \tau_{B(x_0,r)} \leq c_1r^{2\alpha/d}.$$
\end{prop}

\begin{proof}
Again, let us denote $B(x_0,r)$ by $B$, and $\tau_B$ by $\tau$.  Fix $z \in B(x_0,r) \setminus \NN$.  We observe that  
$$\ppz(\tau > t) \leq \ppz(X_t \in B) = \ppz(X_t \in B \setminus \NN) =  \int_{B \setminus \NN} p(t,z,y) \thinspace dy.$$
Therefore, by Proposition \ref{p31}, 
 $$\ppz(\tau > t) \leq \int_{B \setminus \NN} c_2(t^{-d/\alpha}\wedge 1)\thinspace dy \leq c_3r^2t^{-d/\alpha}.$$  
Thus, by taking $c_4$ large enough, there is a time $t_0 = c_4r^{2\alpha/d}$ such that 
\begin{equation}\label{t0}
\ppz(\tau > t_0) \leq {1}/{2}.
\end{equation}

Applying the strong Markov property, we have that 
\begin{align*}
\ppz(\tau > 2t_0) &= \ppz(\tau > t_0, \tau \circ \theta_{t_0} > t_0)\\
&= \eez [\ppz(\tau \circ \theta_{t_0} > t_0 \mid \FF_{t_0})1_{(\tau > t_0)}]\\
&= \eez[[\pp^{X_{t_0}}(\tau > t_0)] ; \tau > t_0]
\\&\leq ({1}/{2}) \ppz(\tau > t_0) < \displaystyle \left({1}/{2}\right)^2,
\end{align*}
since \eqref{t0} is true regardless of the location of $z \in  B(x_0,r) \setminus \NN$.  
In particular then, it follows by an induction argument that for each $m \in \nn$, 
$$\ppz(\tau > mt_0) \leq 2^{-m}.$$

We now can write
\begin{align*}
\eez \tau &= \eez[\tau;\tau < t_0] + \sum_{k=0}^\infty\eez[\tau; 2^kt_0 \leq \tau < 2^{k+1}t_0] \\
&\leq t_0 +  \sum_{k=0}^\infty 2^{k+1}t_02^{-2k} = c_5t_0 = c_6r^{2\alpha/d}.
\end{align*}
Since this bound does not depend on $z$, it holds for $\sup_{z \in (B(x_0,r) \setminus \NN)} \eez \tau$ as well. 
\end{proof}

We end this section with a result that will show that $X_t$ will hit sets of positive Lebesgue measure with positive probability.

\begin{prop}\label{p34}
Suppose  $A  \subseteq B(x_0,R)$.  There exists $c_1$ not depending on $x_0$ or $A$ such that
$$ \pp^x(T_A <\tau_{B(x_0,\delta R)}) \geq c_1|A|, \qquad x \in (B(x_0,\gamma R) \setminus \NN), $$
for any $1<\gamma <\delta$.
\end{prop}

\begin{proof}
Let $B(x_0, \delta R)$ be denoted by $B$.  Theorem \ref{t25} c) states that there is a constant $c_2$, depending on $R$ and $\delta$ such that
\begin{equation}
p^B(t,x,y) \geq c_2, \qquad t \in [1/4,1].
\end{equation} 
Therefore, we obtain that
\begin{align*}
\ppx(T_A < \tau_B) &\geq \ppx(\tau_B > 1, X_1 \in A)  \\
&= \int_A p^B(1,x,y) dy \\
&\geq c_2|A|.
\end{align*}
\end{proof}

\section{Occupation Times}

We now will show that these processes will be expected to spend some positive amount of time in a set having positive Lebesgue measure.  The proof that we will give is an adaptation of the proof of a similar result in the nondivergence case, which was shown in \cite{bass}.  First, we need to prove a fact about the resolvents of these processes.  We define 
$$S_\lambda f(x) = \eex \int_0^\infty e^{-\lambda t}f(X_t)dt.$$

\begin{prop}\label{p41b}
Let $x \in \rr^d \setminus \NN$ and $\lambda > 0$.  For  every $\varepsilon > 0$, there exists $\delta > 0$, such that if $C$ is any Borel set with $|C| < \delta$, then
$|S_\lambda1_C(x)| < \varepsilon. $

\end{prop}

\begin{proof}
Fix $\varepsilon > 0$.  We observe that
\begin{align*}
S_\lambda 1_C(x) &= \eex \int_0^\infty e^{-\lambda t}1_C(X_t)dt \\
&\leq \eex \int_0^{\varepsilon / 2} e^{-\lambda t}1_C(X_t)dt + \eex \int_{\varepsilon/2}^\infty e^{-\lambda t}1_C(X_t)dt \\
&\leq {\varepsilon}/{2} + \int_{\varepsilon/2}^\infty e^{-\lambda t}\int_Cp(t,x,y)\thinspace dy \thinspace dt \\
&\leq {\varepsilon}/{2} + c_1\int_{\varepsilon/2}^\infty e^{-\lambda t}\int_C(t^{-d/\alpha} \wedge 1)\thinspace dy \thinspace dt, 
\end{align*} 
by an application of Proposition \ref{p31}.  Therefore, 
$$|S_\lambda 1_C(x)| \leq  {\varepsilon}/{2} +c_2(\varepsilon/2)^{1 - d/\alpha}|C|.$$
 Now choose $\delta$ so that if $|C| < \delta$, then $c_2(\varepsilon/2)^{1 - d/\alpha}|C| \leq {\varepsilon}/{2}$.
\end{proof}

Now we show that these processes will spend some time in a set which is almost the entire cube $Q(x_0,R)$.

\begin{prop} \label{p42}
There exist $c_1$ and $\varepsilon$ depending on $R$, such that if $B \subseteq Q(x_0,R)$, $x \in Q(x_0,R/2) \setminus \NN$, and $|Q(x_0,R)-B| < \varepsilon$, then  $$\eex \int_0^{\tau_{Q(x_0,R)}}1_B(X_s)ds \geq c_1.$$
\end{prop}

\begin{proof}
Let us denote $\tau_{Q(x_0,R)}$ by $\tau$.  By Proposition \ref{p32}, there exists $c_2$ such that $\eex \tau \geq c_2$, and by Proposition \ref{p33b}, we have that $\sup _x \ee^x \tau \leq c_3$, so that $\eex \tau^2 \leq c_4.$  (Both $c_2$ and $c_4$ will depend on $R$.)

Since 
$$ \eex(\tau - (\tau \wedge t_0)) \leq \eex (\tau ; \tau \geq t_0) \leq \eex \tau^2 / t_0,  $$    
we are able to choose $t_0$ large enough to ensure that $\eex(\tau - (\tau \wedge t_0))  \leq c_2 / 4 $.  Therefore, 
\begin{align*}
\eex \int_0^\tau &1_{(Q(x_0,R) -B)}(X_s)ds \\
& \leq c_2 /4 + e^{t_0}\eex \int_0^{t_0} e^{-s} 1_{(Q(x_0,R) -B)}(X_s)ds \\
& \leq c_2 /4 + e^{t_0}\eex \int_0^\infty e^{-s} 1_{(Q(x_0,R) -B)}(X_s)ds \\
& \leq c_2 /4 + e^{t_0}\eex \int_0^\infty e^{-s\lambda} 1_{(Q(x_0,R) -B)}(X_s)ds. 
\end{align*}
Now by Proposition \ref{p41b}, we can choose $\varepsilon$ small enough so that $$e^{t_0}\eex \int_0^\infty e^{-s\lambda} 1_{(Q(x_0,R) -B)}(X_s)ds < c_2 /4,$$ so this proposition will hold with $c_1 = c_2 / 2$.
\end{proof}

\begin{lemma} \label{l43}
Suppose $r > 1$ and let $W$ be a cube in $Q(x_0,R)$.  Let $W^*$ be the cube with the same center as $W$ but side length half as long.  Let $V$ be a subset of $W$ with the property that there exists $\delta$ such that
$$ \eey \int_0^{\tau_W} 1_V(X_s)ds \geq \delta \eey \tau_W, \qquad y \in W^*\setminus \NN.$$
Then there exists $\xi(\delta)$ depending on $\delta$ and r such that
$$ \eey \int_0^{\tau_{Q(x_0,rR)}} 1_V(X_s)ds \geq \xi(\delta) \eey \int_0^{\tau_{Q(x_0,rR)}} 1_W(X_s)ds, \qquad y \in Q(x_0,R)\setminus \NN. $$
\end{lemma}

\begin{proof}
Let $S$ be the cube in $Q(x_0,rR)$ with the same center as $W$ but side length $r \wedge 2^{1/d}$ as long.  Let $T_1 = \inf \{t: X_t \in W \}$, $U_1 = \inf \{t > T_1 : X_t \notin S \}$, $T_{i+1} = \inf \{t > U_i: X_t \in W \}$, and $U_{i+1} = \inf \{t > T_{i+1} : X_t \notin S \}$.  Then
\begin{align*}
 \eey \int_0^{\tau_{Q(x_0,rR)}}1_W(X_s)ds  &= \sum \eey \Bigl[ \int_{T_i}^{U_i}1_W(X_s)ds ; T_i < \tau_{Q(x_0,rR)}    \Bigr], \\ 
&=  \sum \eey \Bigl[\ee^{X(T_i)}  \int_{0}^{\tau_S}1_W(X_s)ds ; T_i < \tau_{Q(x_0,rR)}    \Bigr],
\end{align*}
and similarly this equation also holds if we replace $W$ by $V$.  Thus we need to show that there exists a $\xi(\delta)$ such that 
$$ \ee^w \int_0^{\tau_S}1_V(X_s)ds \geq \xi(\delta) \ee^w \int_0^{\tau_S} 1_W(X_s)ds, \qquad w \in W \setminus \NN. $$

We observe that ${|W|}/|S| =  ({1}/{r^d}) \vee ({1}/{2})$, a quantity which does not depend on the size of $W$, so by Proposition \ref{p34} there exists $c_1$ only depending on $r$ such that
$$ \pp^w(T_{W^*} < \tau_S) \geq c_1, \qquad w \in W \setminus \NN. $$

So if $w \in W \setminus \NN$, the strong Markov property implies that
\begin{align*}
\ee^w \int_0^{\tau_S} 1_V(X_s)ds &\geq \ee^w \Bigl[  \int_0^{\tau_S} 1_V(X_s)ds ; T_{W^*} < \tau_S \Bigr] \\
&= \ee^w \Bigl[ \ee^{X(T_{W^*})}  \int_0^{\tau_S} 1_V(X_s)ds ; T_{W^*} < \tau_S \Bigr] \\
&\geq c_1 \inf_{z \in W^*\setminus \NN} \eez \int_0^{\tau_S} 1_V(X_s)ds \\
&\geq c_1 \inf_{z \in W^*\setminus \NN} \eez \int_0^{\tau_W} 1_V(X_s)ds. 
\end{align*}

By hypothesis, if $z \in W^*\setminus \NN$, 
$$ \eez \int_0 ^{\tau_W}1_V(X_s) ds \geq \delta\eez \tau_W.$$

By Proposition \ref{p32}, 
$$ \eez \tau_W \geq c_2 \sup_{v \in S \setminus \NN} \ee^v \tau_S  \geq  c_2 \ee^w  \int_0^{\tau_S} 1_W(X_s) ds.$$

Taking $\xi(\delta) = c_1c_2\delta$ completes the proof.  Note that the way which $c_1$ and $c_2$ were chosen implies that neither constant can be greater than one, so we have that $\xi(\delta) \leq \delta.$ 
\end{proof} 

\begin{theorem} \label{t44}
 For every value of $R > 0$, there exists a nondecreasing function $\varphi : (0,1) \to (0,1)$ such that if $x \in Q(x_0, R/2) \setminus \NN$ and $B \subseteq Q(x_0,R)$, then
$$ \eex \int_0^{\tau_{Q(x_0,R)}}1_B(X_s) ds \geq \varphi \displaystyle \left(\frac{|B|}{R^d}\right). $$ 
\end{theorem}

\begin{proof}
Fix $R > 0$, and let
\begin{align*}
\varphi(\varepsilon) = \inf\Bigl\{ \eey \int_0^{\tau_{Q(z_0,R)}}&1_B(X_s)ds \colon z_0 \in \rr^d,B \subseteq Q(z_0,R), \\
 &|B| \geq \varepsilon|Q(z_0,R)|, y \in Q(z_0,R/2)  \setminus \NN \Bigr\}.
\end{align*}
By Proposition \ref{p42} , we obtain that $\varphi(\varepsilon) > 0$ for $\varepsilon$ sufficiently close to 1.  Our goal, then, is to show that $\varphi(\varepsilon) > 0$ for all positive $\varepsilon$.

Let $q_0$ be the infimum of the $\varepsilon$ for which $\varphi(\varepsilon) > 0$.  We will argue by contradiction, and will suppose that $q_0 > 0$.  Since $q_0 < 1$, there exists a $q > q_0$ such that $(q + q^2)/2 <q_0$.  Set $\gamma = (q - q^2)/2$.  Let $\beta$ be a number of the form $2^{-n}$ with
$$ (\gamma \wedge q \wedge(R-q))/32d \leq \beta < (\gamma \wedge q \wedge(R-q))/16d. $$

Since $\xi(\delta) \leq \delta$ and $\varphi$ is an increasing function, there exist $z_0 \in \rr^d$,  $B_1 \subseteq Q(z_0, R)$, and $x \in Q(z_0, R/2) \setminus \NN$ such that $q > |B_1|/|Q(z_0,R)| > q - \gamma / 2$ and 
$$ \eex  \int_0^{\tau_{Q(z_0,R)}}1_{B_1}(X_s)ds < \xi(\varphi(q))\varphi(q),$$
where $\xi$ is defined in Lemma \ref{l43}.  We will prove the result in the special case where $z_0 = x_0$.  The general case can be shown similarly.
Therefore, 
$$\eex  \int_0^{\tau_{Q(x_0,R)}}1_{B_1}(X_s)ds < \xi(\varphi(q))\varphi(q).$$

Let $B = B_1 \cap Q(x_0,R-\beta)$.  Then
$$\eex  \int_0^{\tau_{Q(x_0,R)}}1_{B}(X_s)ds < \xi(\varphi(q))\varphi(q),$$
and $q > |B| > q - \gamma$, since we chose $\beta$ small enough to guarantee that $$|Q(x_0,R) - Q(x_0,R-\beta)| \leq \gamma/2.$$
As in the Harnack inequality proof given by Krylov and Safonov \cite{ks}, we construct $D$ consisting of the union of cubes $\widehat R_i$, such that $$|D \cap Q(x_0,R)| \geq ({|B|}/{q}) > R^d({q - \gamma})/{q} = R^d({q + 1})/{2},$$ and such that $ |B \cap R_i| > q|R_i|$ for all $i$.  We also have that the $R_i$ have pairwise disjoint interiors, where $R_i$ is the cube with the same center as $\widehat R_i$ and one-third the side length.  (For a proof of this, see Chapter 5, Section 7 of \cite{bass}.)  Let $\widetilde D =D \cap Q(x_0,R)$.  Then we see that
$$|\widetilde D|  \geq R^d({q + 1})/{2} > R^dq > R^dq_0,$$
and therefore
$$ \eex \int_0^{\tau_{Q(x_0,R)}}1_{\widetilde D}(X_s)ds > \varphi(q).$$

Let $V_i = \widehat R_i \cap Q(x_0,R-\beta)$.  We want to show for each $i$,
\begin{equation} \label{goalb}
\eex  \int_0^{\tau_{Q(x_0,R)}}1_{B \cap R_i}(X_s)ds \geq \xi(\varphi(q)) \eex  \int_0^{\tau_{Q(x_0,R)}}1_{V_i}(X_s)ds.
\end{equation}

Once we have \eqref{goalb}, we sum and we have 
\begin{align*} 
\eex  \int_0^{\tau_{Q(x_0,R)}}1_{B}(X_s)ds &\geq \sum_i \int_0^{\tau_{Q(x_0,R)}}1_{B \cap R_i}(X_s)ds \\
&\geq \xi(\varphi(q))  \sum_i \eex \int_0^{\tau_{Q(x_0,R)}}1_{V_i}(X_s)ds \\
&\geq \xi(\varphi(q))  \eex \int_0^{\tau_{Q(x_0,R)}}1_{\widetilde D}(X_s)ds \\
&\geq \xi(\varphi(q))\varphi(q),
\end{align*}
which is our desired contradiction.

We now prove \eqref{goalb}.  Fix $i$.  By our definition of $\beta$, if $V_i$ is not empty, then $V_i$ is contained in a cube $W_i$ which is itself a subset of $Q(x_0,R-\beta)$, such that $|W_i| \leq 3^d|R_i|$.  Let $ R_i^*$ be the cube with the same center as $R_i$ but side length half as long.  By the definition of $\varphi$, 
$$ \eey \int_0^{\tau_{R_i}} 1_{B \cap R_i}(X_s)ds \geq \varphi(q)\eey \tau_{R_i} $$
if $ y \in R_i^* \setminus \NN$.  We can now deduce \eqref{goalb} from Lemma \ref{l43}. 
\end{proof}

\section{Support Theorem}\label{supp}
In this section, we will prove a support theorem for $X$.  This proof is somewhat similar to the one given by Bass and Chen \cite{bc2}.  However,  the processes considered in \cite{bc2} can only jump in finitely many different directions, so our proof will require some different techniques.  We begin by proving some lemmas.

\begin{lemma}\label{l21}
Let $x_0, x_1 \in \rr^d \setminus \NN$ with $z = x_1 - x_0$ such that $|z|<1/2$ and fix $0<\gamma <1/4$ and $t_0 \geq 0$.  There exists a stopping time $T$ and a positive constant $c_1$ depending only on $\gamma$, $|z|$, and $t_0$,  such that 
\begin{equation*} 
\pp^{x_0}(T \leq t_0, 
 \displaystyle \sup_{s<T} |X_s - x_0| < \gamma ,
    \sup_{T \leq s \leq t_0} |X_s - x_1| < \gamma) \geq c_1.
\end{equation*}
\end{lemma}

\begin{proof}
Let $\beta = {|z|}/{2}$ , and $\delta = ({\gamma}/{3}) \wedge  ({|z|}/{6})$.  We define $$J^{(\beta)}(x,y) = J(x,y)\mathbbm 1_{(|x-y|< \beta)},$$ and let $\EE^{(\beta)}$ be defined by \eqref{ee} with $J^{(\beta)}$ in place of $J$. 
Using Meyer's construction, which we reference later in this proof, we obtain the existence of a strong Markov process $\overline X$ associated to $(\EE^{(\beta)}, \FF)$ defined on $\rr^d \setminus \NN(\beta)$.
Therefore, $\overline X$ has no jumps having size larger than $\beta$.  In fact, by this construction we can take $\NN(\beta) = \NN$, so $\overline{X}$ is defined up to the same zero capacity set as $X$ was.  See the proof of Proposition 3.9 in \cite{bbck} for further details. 

Let $B = B(x_0, \delta)$.  It follows from Theorem \ref{t25} c) that 
\begin{equation}
\pp^{x_o}(\overline \tau_{B} > t_0) \geq \int_{B(x_0, 3\delta/4)}\overline p^B(t_0, x_0, y)dy \geq c_2\delta^d
\end{equation}
where $\overline p(t,x,y)$ is the transition density for $\overline X_t$ and $\overline \tau_B$ is the exit time for $\overline X_t$.  We are still able to apply Theorem \ref{t25} in this case, since the only change made to our Dirichlet form is eliminating jumps of size larger than $\beta$, instead of eliminating jumps having size larger than 1.
  Let 
$ E = \displaystyle \{ \sup_{s \leq t_0} |\overline X_s - x_0| \leq \delta \}.$  It follows therefore, that $\pp^{x_0}(E) \geq c_3$, where $c_3$ depends on $\delta$ and $t_0$.

We now will use a construction of Meyer to add some large jumps to the process $\overline X_t$, in order to create a process $X_t$ which will be associated to the operator $(\EE, \FF)$.  A reference for this process is Remark 3.4 of \cite{bbck}. Let $U_1$ and $U_2$ be the times of the first two jumps we add to $\overline X$, and define 
$$D =  \{U_1 \leq t_0 < U_2, \Delta X_{U_1} \in B(z,\delta) \}.$$

Let $S_1$ and $S_2$ be independent exponential random variables of parameter 1, which are also independent of $\overline X$.    We note that for any $x \in \rr^d$, 
\begin{align}\label{e27}
c_4 =\int_{\beta \leq |w| <1 } \kappa_1|w|^{-d-\alpha}dw &\leq \int_{|x-y|\geq\beta} J(x,y)dy \notag \\
&\leq \int_{\beta \leq |w| <1 } \kappa_2|w|^{-d-\alpha}dw = c_5. 
\end{align}
We define
$$ F = \{ S_1 \in (0, c_4t_0], S_2\in [c_5t_0, \infty) \}. $$
Since $S_1$ and $S_2$ are exponential,  
\begin{equation*}
\ppx(S_1 \leq c_4t_0) = 1 - e^{-c_4t_0}
\end{equation*} and
\begin{equation*}
\ppx(S_2 \geq c_5t_0) = e^{-c_5t_0},
\end{equation*}
so by the independence of $S_1$ and $S_2$, $\ppx(F)  \geq c_6$, where $c_6$ depends on $t_0$.  Furthermore, the event $F$ was chosen to be independent of $E$, so we have
$$ \pp^{x_0}(F \cap E) =  \pp^{x_0}(F)\ \pp^{x_0}(E) \geq c_3c_6.$$

We define
$$C_r = \int_0^r\int_{|\overline X_s-y|\geq \beta} J(\overline X_s, y)dy \thinspace ds.$$
Using Meyer's construction, we will introduce an additional jump to $\overline X$ at the first time $U_1$ such that $C_{U_1}$ exceeds $S_1$, restart the process, and then introduce a second jump when $C_{U_2}$ exceeds $S_2$.  It follows from \eqref{e27} that if $F$ holds, we will add  exactly one jump to $\overline X$ before time $t_0$, so that if $G = \{U_1 \leq t_0 < U_2\}$,
$$ \pp^{x_0}( G \cap E ) \geq  \pp^{x_0}(F \cap E) \geq c_3c_6.$$

Suppose now that $G$ holds.  
The location of the jump at time $U_1$ of size larger than $\beta$ will be determined by the distribution
$$q(x,dw) = \displaystyle \frac{J(x,w)}{\int_{|x-y| \geq \beta} J(x,y)dy}dw$$
where $\beta \leq |x - z| < 1$.  Therefore, if $B = B(z, \delta)$, 
\begin{align*}
\pp^{x}(\Delta X_{U_1} \in B) &= \int_B q(X_{U_1-},X_{U_1-} + dw)  \\
&\geq \int_B \displaystyle \frac{J(X_{U_1-}, w)}{c_4}dw  \\
&\geq \int_B \frac{\kappa_2}{c_5|w|^{d+\alpha}}dw \geq \int_B\frac{\kappa_2}{c_5[(3/2)|z|]^{d+\alpha}}dw = c_7, 
\end{align*}
a constant which depends only on $\gamma$ and $|z|$.  This bound does not depend on $\overline X_{U_1}$, so we have that $ \pp^{x_0}(D \cap E) \geq c_3c_6c_7$.

We now note that on $D \cap E$, 
$$\displaystyle \sup_{s<U_1} |X_s - x_0| < \delta < \gamma,$$ 
and
$$\displaystyle \sup_{U_1 \leq s \leq t_0} |X_s - x_1| \leq \delta + \delta + \delta <  \gamma,$$ 
so $U_1$ is our desired stopping time $T$.
\end{proof}

\begin{lemma}\label{l22}
Let $t_1 > 0$, $\varepsilon> 0$, $r \in (0, (\varepsilon \wedge 1)/2)$, and $\gamma > 0$.  Let $\psi \colon [0, t_1] \to \rr^d$ be a line segment of length $r$ starting at $x_0$.  Then there exists $c_1 > 0$ that depends only on $t_1$, $\varepsilon$, and $\gamma$ such that
$$\pp^{x_0}\displaystyle \left(\sup_{s\leq t_1} |X_s - \psi(s)| < \varepsilon\textit{ and }|X_{t_1} - \psi(t_1)| < \gamma \right) \geq c_1.$$
\end{lemma}

\begin{proof}

Let $T$ be the stopping time from Lemma \ref{l21}.  Let $D$ be the event that  $|X_s - x_0| < \gamma \wedge\varepsilon/4 \wedge 1/4$ for $s \in [0,T]$ and $|X_s - \psi(t_1)| < \gamma \wedge \varepsilon/4 \wedge 1/4$ for $s \in [T, t_1]$.
 By Lemma \ref{l21}, there exists a constant $c_2 > 0$ such that $\pp^{x_0}(D) \geq c_2$.  We now note that on $D$, by definition, $|X_{t_1} - \psi(t_1)| < \gamma$.   We now show that on $D$,  $|X_s - \psi(s)| < \varepsilon$ for every $s \in [0, t_1]$.

If $s < T$, since $r < \varepsilon/2$, we have that
\begin{equation*}
|X_s - \psi(s)| \leq |X_s -x_0| + |x_0 - \psi(s)| < {\varepsilon}/{4} + {\varepsilon}/{2} < \varepsilon. 
\end{equation*}
Similarly, if $T \leq s \leq t_1$, we obtain that
\begin{equation*}
|X_s - \psi(s)| \leq |X_s -\psi(t_1)| + |\psi(t_1) - \psi(s)| < {\varepsilon}/{4} + {\varepsilon}/{2} < \varepsilon. 
\end{equation*}

Therefore, this lemma holds with $c_1 = c_2$.
\end{proof}

\begin{theorem}\label{support}
Let $\varphi : [0,t_0] \to \rr^d$ be continuous with $\varphi(0)= x_0$.  Let $\varepsilon > 0$.  There exists $c_1 > 0$ depending on $\varphi$, $\varepsilon$, and $t_0$ such that
$$\pp^{x_0}\left(\displaystyle \sup_{s\leq t_0} |X_s - \varphi(s)| < \varepsilon \right) > c_1.$$
\end{theorem}

\begin{proof}
We may approximate $\varphi$ to within $\varepsilon/2$ by a polygonal path, so by changing $\varepsilon$ to $\varepsilon/2$, we may without loss of generality assume that $\varphi$ is polygonal.  We now choose $n$ large and subdivide the interval $[0,t_0]$ into $n$ subintervals, so that over each subinterval $[kt_0/n, (k+1)t_0/n]$  the image of $\varphi$  is a line segment whose length is smaller than $\varepsilon/4 \wedge 1/2$.  By Lemma \ref{l22}, there exists $c_2 >0$,  such that on each time interval  $[kt_0/n, (k+1)t_0/n]$,
\begin{align*} 
\pp^{x_0}\displaystyle \Big(\sup_{kt_0/n\leq s\leq (k+1)t_0/n} &|X_s - \psi(s)| < {\varepsilon}/{2}\text{ and } \\ &|X_{(k+1)t_0/n} - \psi((k+1)t_0/n)| < {\varepsilon}/{(4\sqrt{d})} \Big) \geq c_2.
\end{align*}
Now by applying the Markov property $n$ times, we obtain our support theorem.
\end{proof}

\cite{bbck} and \cite{bkk} consider the case that there exist $0 < \alpha < \beta < 2$ and positive constants $\kappa_1$ and $\kappa_2$, such that if $|x-y| < 1$, then
$\kappa_1|y-x|^{-d-\alpha} \leq J(x,y) \leq  \kappa_2|y-x|^{-d-\beta}.$  We remark here that a virtually identical proof will prove the support theorem under that assumption.





\bibliographystyle{model1b-num-names}
\bibliography{<your-bib-database>}

\begin{thebibliography}{20}


\bibitem{bbck}
M. T. Barlow, R. F. Bass, Z.-Q. Chen, M. Kassmann, Non-local Dirichlet forms and symmetric jump processes,
{Trans. Amer. Math. Soc.} {361} (2009) 1963 - 1999.

\bibitem{bass}
R. F. Bass, {Diffusions and Elliptic Operators},  Springer, New York, 1997. 

\bibitem{bc2} 
R. F. Bass, Z.-Q. Chen, Regularity of harmonic functions for a class of singular stable-like processes, Math. Z. {266} (2010) 489-503.

\bibitem{bkk}
R. F. Bass, M. Kassmann, T. Kumagai, Symmetric jump processes: localization, heat kernels, and convergence,  
{Ann. de l'Institut H. Poincar\'e} 46 (2010) 59-71. 

\bibitem{fot}  M. Fukushima, Y. Oshima, M. Takeda, {Dirichlet Forms and Symmetric Markov Processes}, volume 19 of  de Gruyter Studies in Mathematics,  Walter de Gruyter \& Co., Berlin, 1994.

\bibitem{ks}
N.V. Krylov, M.V. Safonov, A property of the solutions of parabolic equations with measurable coefficients, Izv. Akad. Nauk SSSR Ser. Mat. {44} (1980) 161-175 and 239. 

\bibitem{ss} 
R. Schneider, O. Reichmann, C. Schwab, Wavelet solution of variable order pseudodifferential equations, Calcolo 47 , 2 (2010) 65-101. 


\bibitem{bmw}
B. M. Whitehead, Occupation times for stable-like processes, Potential Anal., to appear.

 \end{thebibliography}



\end{document}